\documentclass[11pt,english,reqno]{amsart}
\numberwithin{equation}{section}
\usepackage{bm}
\usepackage{color}
\usepackage{babel}
\usepackage{verbatim}
\usepackage{amsthm}
\usepackage{amstext}
\usepackage{amssymb}
\usepackage{amsbsy}
\usepackage{ifpdf}
\usepackage[nobysame,abbrev,alphabetic]{amsrefs}
\usepackage{enumerate}
\usepackage{amsmath}
\usepackage{amscd}
\usepackage{amsfonts}
\usepackage{graphicx}
\usepackage[all]{xy}
\usepackage{verbatim}
\usepackage{gensymb}
\usepackage{mathrsfs}
\usepackage[colorlinks]{hyperref}
\hypersetup{
linkcolor=blue,          
citecolor=green,        
}
\newtheorem{theorem}{Theorem}[section]

\newtheorem{lemma}[theorem]{Lemma}
\newtheorem{corollary}[theorem]{Corollary}
\theoremstyle{definition}\newtheorem{definition}[theorem]{Definition}

\theoremstyle{theorem}\newtheorem{proposition}[theorem]{Proposition}

\newtheorem{problem}[theorem]{Problem}
\theoremstyle{definition}

\theoremstyle{definition}
\theoremstyle{definition}\newtheorem{remark}[theorem]{Remark}
\theoremstyle{definition}\newtheorem*{acknowledgments}{Acknowledgments}
\newcommand{\al}{\alpha}

\newcommand{\ga}{\gamma}
\newcommand{\Ga}{\Gamma}
\newcommand{\del}{\delta}

\newcommand{\lam}{\lambda}
\newcommand{\Lam}{\Lambda}
\newcommand{\eps}{\epsilon}

\newcommand{\Sig}{\Sigma}

\newcommand{\vphi}{\varphi}
\newcommand{\cA}{\mathcal{A}}

\newcommand{\cG}{\mathcal{G}}

\newcommand{\cS}{\mathcal{S}}

\newcommand{\bR}{\mathbb{R}}
\newcommand{\bZ}{\mathbb{Z}}

\newcommand{\bT}{\mathbb{T}}

\newcommand{\SL}{\operatorname{SL}}

\newcommand{\GL}{\operatorname{GL}}

\newcommand{\defi}{\overset{\on{def}}{=}}
\newcommand{\comp}{\textrm{{\tiny $\circ$}}}

\newcommand\norm[1]{||#1||}

\newcommand\set[1]{\left\{#1\right\}}
\newcommand\pa[1]{\left(#1\right)}
\newcommand\idist[1]{\langle#1\rangle}

\newcommand\av[1]{|#1|}
\newcommand\on[1]{\operatorname{#1}}

\newcommand\mb[1]{\mathbf{#1}}

\newcommand\smallmat[1]{\pa{\begin{smallmatrix}#1\end{smallmatrix}}}

\newcommand{\Mat}{\operatorname{Mat}}

\newcommand{\wstar}{\overset{\on{w}^*}{\lra}}

\newcommand{\supp}{\on{supp}}

\newcommand\diam[1]{\on{diam}{(#1)}}
\newcommand\ol[1]{\overline{#1}}

\newcommand{\Hom}{\on{Hom}_{\on{epi}}(\bZ^m,\bZ^n/\Sig)}
\newcommand{\gen}{\cS_{\bZ^n/\Sig}(m)}

\newcommand{\lra}{\longrightarrow}

\newcommand{\onto}{\xymatrix{\ar@{>>}[r]&}}

\newcommand{\eqlabel}[2]
{
\begin{equation}
{#2}\label{#1}
\end{equation}
}

\begin{document}
\title[Random Cayley graphs]{Asymptotic metric behavior of random
Cayley graphs of finite abelian groups}
\author[Uri Shapira]{Uri Shapira}
\author[Reut Zuck]{Reut Zuck}

\begin{abstract}
Using methods of Marklof and Str\"ombergsson we 
establish several limit laws for metric parameters of random Cayley graphs of 
finite abelian groups with respect to a randomly chosen set of generators of a fixed size. Doing so we settle a 
conjecture of Amir and Gurel-Gurevich.
\end{abstract}
\address{Department of Mathematics\\
Technion \\
Haifa \\
Israel }
\email{ushapira@tx.technion.ac.il}
\email{reut@tx.technion.ac.il}
%
%
\maketitle

\section{\label{sec:Introduction}Introduction}

\subsection{The main result}
For a finite group $\Ga$ and a generating set $s\subset \Ga$ we denote by $C_\Ga(s)$ the corresponding Cayley graph and 
by $C_\Ga^+(s)$ the corresponding directed Cayley graph (digraph for short). 
Let $m\ge n$ be integers with $m\ge 2$ and let $\Sig<\bZ^n$ be a finite index subgroup. Consider the set
$$\cS_{\bZ^n/\Sig}(m)\defi\set{s:s\subset\bZ^n/\Sig, \av{s}=m, \idist{s}=\bZ^n/\Sig}.$$
Our aim is to establish several limit laws for (proper scalings of) some metric parameters on random Cayley graphs or digraphs
which are obtained by choosing $s\in\cS_{\bZ^n/\Sig}(m)$ at random, as $\Sig\to\infty$ in the following sense. 
\begin{definition}\label{def:divergence} 
If $m>n$ we say that $\Sigma\to\infty$
if $\av{\Sig}\to\infty$, where we denote by $\av{L}$ the covolume of a lattice $L<\bR^n$. 
If $n=m$, we say that $\Sig\to\infty$ if $\av{\Sig}/(\gcd\Sig)^n\to\infty$, where $\gcd \Sig$ denotes the largest positive integer $\ell$ for which $\ell^{-1}\Sig\subset \bZ^n$ (in case $m=n$, we basically want to exclude the sequence $\Sig=k\bZ^n$, where $\av{\Sig}=(\gcd\Sig)^n$).
\end{definition} 
Our results will refer to a metric parameter $\xi(\cG)$ of a graph or digraph according to the following list. We 
will refer to the choice $\cG = C_{\bZ^n/\Sig}(s)$ as the \textit{undirected case} and to $\cG =C_{\bZ^n/\Sig}^+(s)$ as the \textit{directed case}:
\begin{enumerate}[(I)]
\item\label{cdiam} $\xi(\cG)=\on{diam}(\cG)$ (both directed and undirected cases),
\item\label{cavgdist} $\xi(\cG)^\al=\frac{1}{\av{\on{V}_\cG}^2}\sum_{x,y\in \on{V}_\cG}\on{d}_\cG^\al(x,y)$, for $\al>0$  (both directed and undirected cases),
\item\label{cgirth} $\xi(\cG) = \on{girth}(\cG)$ (only in the directed case).
\end{enumerate}
Above $\on{V}_\cG$ denotes the set of vertices of $\cG$ and $\on{d}_\cG$ denotes the graph metric in the undirected case and non-symmetric metric (ns-metric for short), in the directed case. 
\begin{theorem}\label{main theorem 1}
Let $\xi$ be as in~\eqref{cdiam}, \eqref{cavgdist}, or \eqref{cgirth}. Then, as $\Sig\to\infty$ in the sense of Definition~\ref{def:divergence}, 
the random variable $s\mapsto \av{\Sig}^{-\frac{1}{m}} \xi(\cG)$, defined 
on $\cS_{\bZ^n/\Sig}(m)$ (where $\cG = C_{\bZ^n/\Sig}(s)$ or $\cG = C_{\bZ^n/\Sig}^+(s)$ according 
to the case, and $s$ is chosen uniformly at random), converges in distribution 
and the limit distribution is given by an explicit random variable (which depends on $\xi$) 
on the space of unimodular lattices
in $\bR^m$ equipped with the natural $\SL_m(\bR)$-invariant probability measure.
\end{theorem}
The explicit description 
of the random variables giving the limit distributions will be given as soon as we introduce the necessary terminology at 
the end of \S\ref{sec:ld}. 

If one restricts attention to the case $n=1$ and $\xi$ as in~\eqref{cdiam} then Theorem~\ref{main theorem 1} settles~\cite[Conjecture 3]{AGG}. 

Our methods are adaptations of those introduced by Marklof and Str\"ombergsson
in~\cite{MS11}. In some respects we improve on their results where the main input which allows us to do so 
is an observation that certain collections of sublattices of $\bZ^m$ that appear naturally in the discussion are invariant under
the action of $\SL_m(\bZ)$ (similar invariance was noticed and used in~\cite{EMSS}). 
We then invoke equidistribution results for such invariant sets (\cite{GM03, COU01, EO06}). 

We now wish to highlight the differences between our results and the results in~\cite{MS11}. 
The first difference is that in~\cite{MS11} 
only the case $n=1, m\ge 2$ is considered, namely their random graphs are Cayley graphs of $\bZ/k\bZ$.  
The second (and most significant difference) is that in \cite{MS11} the random graph is not 
chosen from $\set{C_{\bZ/k\bZ}(s): s\in \cS_{\bZ/k\bZ}(m)}$ for a given $k$ but from the union 
$\cup_{\ell\le k}\set{C_{\bZ/\ell\bZ}(s): s\in \cS_{\bZ/\ell\bZ}(m)}$ (i.e.\ the group with respect to which the random Cayley
graph is chosen is not fixed), 
which is less natural from the graph theoretic point of view. On the other hand, the discussion in~\cite{MS11} allows to 
choose the generating set $s$ at random with respect to some non-uniform measures on $\cup_{\ell\le k}\cS_{\bZ/\ell\bZ}(m)$, as opposed to 
the statement in Theorem~\ref{main theorem 1} where $s$ is chosen uniformly from $\cS_{\bZ^n/\Sig}(m)$. 
The input that allows~\cite{MS11} to work with non-uniform measures is that in their work they apply a more sophisticated 
equidistribution result from~\cite{M10} than Theorem~\ref{thm:eqdist} that we are using. It is that input that forces them to randomly choose
the group and the generating set and not just the generating set. Recently the equidistribution result from~\cite{M10} was generalized in~\cite{EMSS}.
In~\S\ref{refinements} we prove Theorem~\ref{main theorem 2} which generalize 
Theorem~\ref{main theorem 1} under certain restrictions so that the generating set $s$ would be chosen 
according to a non-uniform measure on $\cS_{\bZ^n/\Sig}(m)$. Similarly to \cite{MS11} the limit distribution does not depend on the law with respect to 
which we choose $s$. The restrictions under which this stronger result holds are that $m>n$ and $\Sig = k\bZ^n$ so that the parameter going to $\infty$
is $k$. If $n=1$ as in~\cite{MS11} this is the general case but in general we do not know at the moment what happens in the case $n=m$ or for general $\Sig$.

We wish to suggest the following problem which is natural from the point of view taken in this paper and leads to an interesting 
equidistribution question in the space of lattices.
\begin{problem}
Let $m>n$, and let $\mb{b}= \set{v_i}_1^n$ be  a basis for $\bZ^n$. For a finite index subgroup  $\Sig<\bZ^n$ 
consider the subset $\cS_{\bZ^n/\Sig}(\mb{b},m) =\set{ s\in \cS_{\bZ^n/\Sig}(m): s\supset \mb{b} \on{mod}\Sig}$ (that is, we restrict attention to the generating sets which contain the reduction of $\mb{b}$ modulo $\Sig$). Is it true that Theorem~\ref{main theorem 1} holds when instead of choosing $s$ at random from $\cS_{\bZ^n/\Sig}(m)$, we choose 
$s$ at random from $\cS_{\bZ^n/\Sig}(\mb{b},m)$.
\end{problem}

\subsection{Outline of the proof}\label{section:outline}
We begin with some notation.
We denote by $X_m$ the space of unimodular lattices on $\bR^m$ (i.e. lattices of covolume 1) and equip it with the natural `uniform' 
probability measure $m_{X_m}$; that is, the unique $\SL_m(\bR)$-invariant Borel probability measure on it. In our discussion we will
often encounter lattices $\Lam< \bR^m$ which are not unimodular. We then denote by  $\bar{\Lam}\in X_m$
the properly scaled lattice $\av{\Lam}^{-1/m}\Lam$. 

The structure of the proof is as follows. 
We will define a map whose fibers have constant cardinality $s\mapsto \Lam_s$ from
$\cS_{\bZ^n/\Sig}(m)$ to the space of subgroups of $\bZ^m$ of index $\av{\Sig}$ with the following two
fundamental properties:
\begin{enumerate}
\item\label{pr2} The Cayley graphs $C_{\bZ^n/\Sig}(s)$ and $C_{\bZ^m/\Lam_s}(I)$ (and the corresponding digraphs) are isomorphic, where $I$ denotes the standard generators of $\bZ^m$.  
\item\label{pr1} The finite collection $\set{\bar{\Lam}_s: s\in \cS_{\bZ^n/\Sig}(m)}$ becomes equidistributed 
in $X_m$ as $\Sig\to\infty$.
\end{enumerate}
The isomorphism in~\eqref{pr2} means that we can study the metric properties of the random graphs $C_{\bZ^m/\Lam_s}(I)$ instead. The 
advantage of these Cayley graphs, which we refer to as \textit{approximate torus graphs}, is that since the generating set is fixed, they are basically a discrete version of the continuous torus $\bR^m/\Lam_s$. We are therefore reduced to study the relevant metric properties of the random torus $\bR^m/\Lam_s$. As the metric parameters we are interested in are scaled in a simple fashion, 
we are further reduced to study the metric properties of a random torus $\bR^m/\bar{\Lam}_s$ as $s\in\cS_{\bZ^n/\Sig}$ is chosen 
uniformly at random. At this point, property~\eqref{pr1}
kicks in and tells us that this random torus is basically a random torus of the form $\bR^m/L$ where $L\in X_m$ is chosen at
random with respect to $m_{X_m}$.
\begin{acknowledgments}
We would like to thank Jens Marklof, Manfred Einsiedler, Shahar Mozes, and Nimish Shah, for useful discussions.
We would also like to thank the anonymous referees for spotting some inaccuracies and helping to improve the paper.
The authors acknowledge the support of ISF grant 357/13.
\end{acknowledgments}

\section{Preparations}
\subsection{The limit distributions}\label{sec:ld} 
As stated in Theorem~\ref{main theorem 1}, the limit distributions that appear 
in our discussion will be given in terms of functions $\zeta:X_m\to\bR$ which we now turn to discuss. Let $\mb{B}$ denote
the following subset of $\bR^m$ or $\bR^m_+\defi\set{v\in\bR^m: \forall i\; v_i>0}$ according to the case,
\begin{enumerate}
\item\label{balld} $\mb{B} =\set{v\in\bR^m : \sum_1^m \av{v_i}<1}$  in the undirected case,
\item\label{ballud} $\mb{B} =\set{v\in\bR^m_+ : \sum_1^m v_i<1}$ in the directed case.
\end{enumerate}
This is nothing but the unit ball (or its intersection with $\bR^m_+$) with respect to the $\ell_1$-metric. For $x,y\in\bR^n$ we
define 
\begin{equation}\label{eq:distance}
\on{d}_{\bR^m}(x,y) = \inf\set{t>0: y\in x+t\mb{B}}.
\end{equation}
 In the undirected case~\eqref{balld} this is the usual $\ell_1$-distance
from $x$ to $y$ and in the directed case~\eqref{ballud} this is a non symmetric distance which is finite only if $y\in x+\bR^m_+$. Given a lattice $\Lam<\bR^m$ the (ns-)distance $\on{d}_{\bR^m}$ descends to a finite (ns-)distance on the torus 
$\bR^m/\Lam$ defined by 
\begin{equation}\label{torusdistance}
\on{d}_{\bR^m/\Lam}(x+\Lam,y+\Lam) = \inf\set{ \on{d}_{\bR^m}(x,y+v):v\in\Lam}.
\end{equation} 
In fact, the quantity on the right is a minimum due to the discreteness of $\Lam$.

Corresponding to the choice~\eqref{cdiam}-\eqref{cgirth} of $\xi$ in Theorem~\ref{main theorem 1} we have the following list of functions $\zeta$ on lattices $\Lam< \bR^m$:
\begin{enumerate}[(A)]
\item\label{covrad} If $\xi$ is chosen as in~\eqref{cdiam} we shall denote 
$$\zeta(\Lam)=\inf\set{t>0 : \Lam+t\mb{B} = \bR^m}$$ 
where $\mb{B}$ is chosen as in~\eqref{balld} or \eqref{ballud} according to whether we are in the undirected 
or the directed case respectively. This is the \textit{covering radius} of $\Lam$ with respect to $\mb{B}$.
\item\label{avgdisttorus} If $\xi$ is chosen as in~\eqref{cavgdist} with $\al>0$ fixed, 
we shall denote 
\begin{align*}
\zeta(\Lam)^\al &= \frac{1}{(\lam(\bR^m/\Lam))^2}\int_{(\bR^m/\Lam)^2 }\on{d}_{\bR^m/\Lam}^\al(x,y)d\lam(x)d\lam(y)\\
&=\frac{1}{\lam(\bR^m/\Lam)}\int_{\bR^m/\Lam}\on{d}_{\bR^m/\Lam}^\al(0,x)d\lam(x),
\end{align*}
where
$\lam$ is the Lebegues measure on $\bR^m$ and $\on{d}_{\bR^m/\Lam}$ is the distance (resp.\ ns-distance) defined
above according to the case.
\item\label{girthtorus} If $\xi$ is chosen as in~\eqref{cgirth}, we shall denote 
$$\zeta(\Lam) = \inf\set{t>0 : t\mb{B} \cap\Lam\ne\set{0}},$$
where $\mb{B}$ is as in~\eqref{ballud}. This is the \textit{length of the shortest vector} in $\Lam$ with respect to the
ns-distance $\on{d}_{\bR^m}$.
\end{enumerate}
In Theorem~\ref{main theorem 1} the random variables converge in distribution to the random variable $\zeta$ on $(X_m,m_{X_m})$, where $\zeta$ is as in~\eqref{covrad}-\eqref{girthtorus} according to the choice of $\xi$ as in~\eqref{cdiam}-\eqref{cgirth}.

\subsection{Topological and measure theoretic notions}
We review some relevant notions and prove two lemmas that will be used in the proof.
Given a locally compact Hausdorff space $X$, and a sequence of Borel probability measures on it, $\mu_i$, we say that 
$\mu_i$ converges in the weak-star topology to a Borel probability measure $\mu$ on $X$ and denote $\mu_i\wstar \mu$ if for any
continuous function with compact support $f\in C_c(X)$, $\int fd\mu_i\to\int fd\mu$. 

A continuous function $\zeta:X\to\bR$ is said to be \textit{proper} if 
the pre-image of compact sets is compact. The importance of properness to our discussion is that if $\zeta$ is proper and 
$f\in C_c(\bR)$ then $f\comp\zeta\in C_c(X)$. This allows us to translate weak-star convergence of measures on $X$ to weak-star
convergence of measures on $\bR$ which is exactly the notion of convergence in distribution appearing in Theorem~\ref{main theorem 1}. This will be used in our discussion in the form of the following lemma.
\begin{lemma}\label{lem:pushing convergence}
Let $X$ be a locally compact Hausdorff space and let $\mu_i\wstar \mu$ be a converging sequence of Borel  
probability measures on $X$. Let $\zeta:X\to(0,\infty)$ be a continuous proper function and let $\xi_i:\supp{\mu_i}\to(0,\infty)$ be 
measurable functions such that on every compact set $K\subset X$, $\sup\set{\av{\zeta(x)-\xi_i(x)}:x\in K\cap\supp{\mu_i}}\to 0$
as $i\to\infty$. Then $(\xi_i)_*\mu_i\wstar \zeta_*\mu$ or in other words, the random variables $\xi_i$ on $(X,\mu_i)$ converge 
in distribution to the random variable $\zeta$ on $(X,\mu)$.
\end{lemma}
\begin{proof}
Let $f\in C_c((0,\infty))$. Our goal is to show that $\av{\int_X f(\xi_i(x))d\mu_i(x)-\int_Xf(\zeta(x))d\mu(x)}\to 0$. 
To this end, let $\eps>0$ and 
let $\del>0$ be such that for $t,s\in(0,\infty)$ with $\av{t-s}<\del$ we have $\av{f(t)-f(s)}<\eps$. Let $K\subset X$ be a 
compact set containing $\zeta^{-1}(\supp{f})$ (which is compact by the properness assumption), such that for all 
large enough $i$, $\mu_i(X\smallsetminus K)<\norm{f}^{-1}\eps$. The existence of such a set follows from the fact that $\mu$
is a Borel probability measure and that $\mu_i\wstar \mu$.
Our assumptions imply that for all large enough $i$ we also have, $ \av{\zeta(x)-\xi_i(x)}<\del$ for $x\in K\cap\supp{\mu_i}$. We therefore have for
such $i$,
$$\av{\int_X f(\xi_i(x))d\mu_i(x)-\int_Xf(\zeta(x))d\mu_i(x)}\le \int_{K}\eps  d\mu_i + 
\int_{X\smallsetminus K} \norm{f}_\infty d \mu_i\le 2\eps.$$
Finally, by the assumed convergence $\mu_i\wstar \mu$ and the properness
of $\zeta$ we conclude that $f\comp \zeta\in C_c(X)$ and that for all large enough $i$, 
$\av{\int_X f(\zeta(x))d\mu_i(x)-\int_Xf(\zeta(x))d\mu(x)}<\eps$. Together this gives that for all large enough $i$, 
$\av{\int_X f(\xi_i(x))d\mu_i(x)-\int_Xf(\zeta(x))d\mu(x)}<3\eps$ which concludes the proof.
\end{proof}
At some point in the proof we will need to use the equicontinuity of a family of functions which 
are induced from the function $f(x) = \on{d}_{\bR^m}^\al(0,x)$, where $\on{d}_{\bR^m}$ is either the distance or the ns-distance 
introduced in \S\eqref{eq:distance} which is defined for $x$ in $\bR^m$ or $\bR^m_+$ respectively. The following lemma achieves that.
\begin{lemma}\label{equicontinuous}
Let $f$ be a continuous non-negative function on either $\bR^m$ or $\bR^m_+$  such that $\lim_{x\to\infty}f(x) = \infty$. 
For a lattice $\Lam<\bR^m$ define a 
$\Lam$-invariant function on $\bR^m$ by $f_{\bR^m/\Lam}(x)= \min\set{f(x+v):v\in\Lam}$. Then, for any compact set 
$K\subset \bR^m$ or $\bR^m_+$,
the family of restrictions $\set{ f_{\bR^m/\Lam}|_K}$ (as $\Lam$ varies), is equicontinuous.
\end{lemma}
\begin{proof}
Let $K'\supset K$ be a larger compact set to be chosen shortly. 
Given $\eps>0$ by the uniform continuity of $f$ on $K'$ we may choose 
$\del>0$ such that if $x,y\in K'$ are such that $\norm{x-y}\le \del$, then $\av{f(x)-f(y)}<\eps$.

Let $\Lam$ be a lattice and let $x,y\in K$ be such that $\norm{x-y}\le \del$. Assume without loss of generality that $f_{\bR^m/\Lam}(x)\le f_{\bR^m/\Lam}(y)$ 
and let $v_x\in\Lam$ by such that $f_{\bR^m/\Lam}(x)=f(x+v_x)$. Because $f$ diverges at $\infty$ we know that if $K'$ is chosen 
large enough then $K+v_x\subset K'$ and therefore $f_{\bR^m/\Lam}(x)\le f_{\bR^m/\Lam}(y)\le f(y+v_x)\le f(x+v_x)+\eps = f_{\bR^m/\Lam}(x)+\eps$.
\end{proof}

\section{Approximated tori graphs}
In light of the outline described in \S\ref{section:outline}, we turn now to discuss a family of graphs we refer to as approximated tori graphs. The goal of this section is to prove Corollary~\ref{cor:1109} which isolates an important component in the proof of 
Theorem~\ref{main theorem 1}.
\begin{definition}\label{def:apptor}
Let $\Lam<\bZ^m$ be a finite index subgroup and let $I$ be the standard basis of $\bR^m$. 
The \textit{approximate torus graph} (resp.\ \textit{digraph}) associated to $\Lam$ is defined to be the Cayley graph $C_{\bZ^m/\Lam}(I)$ (resp.\
Cayley digraph $C_{\bZ^m/\Lam}^+(I)$).
\end{definition}
We say that two (ns-)metric spaces $(X,\on{d}_X),(Y,\on{d}_Y)$ are \textit{of bounded distance} 
if there are maps $\vphi:X\to Y,\psi:Y\to X$
and a constant $C$ such that for all $x_1,x_2\in X$, $\av{\on{d}_X(x_1,x_2)-\on{d}_Y(\vphi(x_1),\vphi(x_2))}\le C$ and 
for all $y_1,y_2\in Y$, $\av{\on{d}_Y(y_1,y_2) - \on{d}_X(\psi(y_1),\psi(y_2))}\le C$. 
We refer to $C$ as a \textit{distance bound}.
The reason for the terminology in Definition~\ref{def:apptor} is
the following lemma which is left to the reader (cf.\ \cite[equation (2.14)]{MS11}).
\begin{lemma}\label{lem:almost isometric}
As $\Lam<\bZ^m$ ranges over the finite index subgroup, the (ns-)metric spaces $(\cG,\on{d}_\cG)$, 
$(\bR^m/\Lam,\on{d}_{\bR^m/\Lam})$ are of bounded distance with a uniform distance bound 
which depends only on the dimension $m$. Here $\cG=C_{\bZ^m/\Lam}(I)$ or $C_{\bZ^m/\Lam}^+(I)$ and 
$\on{d}_{\bZ^m/\Lam}$ is the distance or ns-distance defined in~\eqref{torusdistance} respectively.
\end{lemma}
The next two lemmas show that the metric parameters in~\eqref{cdiam}, \eqref{cavgdist}, \eqref{cgirth}, appearing 
in Theorem~\ref{main theorem 1}
of an approximated torus graph or digraph $\cG$,
are naturally linked with the corresponding quantities in~\eqref{covrad}, \eqref{avgdisttorus}, \eqref{girthtorus}, 
of the lattice giving rise to 
the approximated torus graph or digraph. 
The first lemma deals with the diameter and the girth, and the second lemma, which is 
more subtle deals with
the average distance and its moments.
\begin{lemma}\label{lem:quazi isometric}
Let $\xi$ be as in~\eqref{cdiam} or \eqref{cgirth} and $\zeta$ be as in~\eqref{covrad} or \eqref{girthtorus} respectively.
Then, as
$\Lam<\bZ^m$ ranges over the finite index subgroup,  
\begin{align}
\label{eq:1215}&\av{\xi(\cG) - \zeta(\Lam)} = O(1),\\
\label{eq:1216}&\av{\av{\Lam}^{-1/m}\xi(\cG) - \zeta(\bar{\Lam})}=O(\av{\Lam}^{-1/m}),
\end{align}
 where $\cG=C_{\bZ^m/\Lam}(I)$ or $C_{\bZ^m/\Lam}^+(I)$ according to the case.
\end{lemma}
\begin{proof}
First note that the quantities $\zeta$ that we consider are homogeneous in the sense that $\zeta(t\Lam) = t\zeta(\Lam)$ and therefore
\eqref{eq:1216} follows immediately from \eqref{eq:1215}.

Regarding case~\eqref{cdiam}: It is straightforward from the definition of the covering radius that
$\diam{\bR^m/\Lam}= \zeta(\Lam)$. Lemma~\ref{lem:almost isometric} implies that $\diam{\bR^m/\Lam}=\diam{\cG}+O(1)$, 
where $\cG =C_{\bZ^m/\Lam}(I)$ or $C_{\bZ^m/\Lam}^+(I)$ according to the case.
The validity of~\eqref{eq:1215} follows. 

Regarding case~\eqref{cgirth}: It is straightforward to show that in fact $\zeta(\Lam) = \on{girth}(C_{\bZ^m/\Lam}^+(I))$ and so~\eqref{eq:1215} holds with $O(1)$ replaced by zero. 
\end{proof}
Before stating the next lemma we introduce some notation. Given $\Lam<\bZ^m$ we denote by $\nu_{\Lam}$ the normalized counting measure supported on the finite set $\bZ^m/\Lam$ viewed as a subset of the torus $\bR^m/\Lam$. Let $\bar{\nu}_\Lam$
denote the image measure on the scaled torus $\bR^m/\bar{\Lam}$. Because of  
the scaling properties of the $\al$'th power of the (ns)-distances we have the following equality
which expresses $\xi(\cG)$ for $\xi$ as in~\eqref{cavgdist} in a form that resembles $\zeta$ as in~\eqref{avgdisttorus}. 
Here $\cG$ is $C_{\bZ^m/\Lam}(I)$ or $C_{\bZ^m/\Lam}^+(I)$ according to the case.
\begin{align}\label{eq:1120}
\xi(\cG)^\al&=\av{\on{V}_\cG}^{-2}\sum_{x,y\in \on{V}_\cG }\on{d}_\cG^\al(x,y) = \int_{(\bR^m/\Lam)^2}\on{d}_{\bR^m/\Lam}^\al(x,y)d\nu_\Lam(x)d\nu_\Lam(y)\\
\nonumber &=\int_{\bR^m/\Lam}\on{d}_{\bR^m/\Lam}^\al(0,x)d\nu_\Lam(x)
=\av{\Lam}^{\frac{\al}{m}}\int_{\bR^m/\bar{\Lam}}\on{d}_{\bR^m/\bar{\Lam}}^\al(0,x) d\bar{\nu}_\Lam(x).
\end{align}
\begin{lemma}\label{lem:moments}
Let $\xi$ be as in~\eqref{cavgdist} and $\zeta$ as in~\eqref{avgdisttorus}. 
Given a compact subset $K\subset X_m$, as $\av{\Lam}\to\infty$, where $\Lam<\bZ^m$ satisfies $\bar{\Lam}\in K$, one has
that $\av{\av{\Lam}^{-\frac{1}{m}}\xi(\cG) - 
\zeta(\bar{\Lam})}\to 0,$
where $\cG$ is $C_{\bZ^m/\Lam}(I)$ (resp.\ $C_{\bZ^m/\Lam}^+(I)$) and $\on{d}_{\bR^m/\bar{\Lam}}$ is the distance 
(resp. ns-distance) defined in~\eqref{torusdistance} according to the case.
\end{lemma}
\begin{proof}
Using~\eqref{eq:1120} we are reduced to showing 
\begin{equation}\label{eq:1131}
\av{\int_{\bR^m/\bar{\Lam}} \on{d}_{\bR^m/\bar{\Lam}}^\al(0,x) d\lam -\int_{\bR^m/\bar{\Lam}} \on{d}_{\bR^m/\bar{\Lam}}^\al(0,x) d \bar{\nu}_\Lam }\to 0.
\end{equation} 
Note that strictly speaking we should have raised the integrals in~\eqref{eq:1131} to a power of $1/\al$ but since $\zeta$ is a proper function the integral on the left
is bounded in terms of $K$ and so it is enough to establish~\eqref{eq:1131} as it is.

We start by finding a convenient fundamental domain for $\bar{\Lam}$ that will allow us to evaluate the left hand side
of~\eqref{eq:1131} and will take advantage of the assumption that $\bar{\Lam}\in K$.
Throughout we use the bar notation to denote scaling by $\av{\Lam}^{-1/m}$. The assumption that $\bar{\Lam}\in K$ means 
that there is a fixed compact set $K'\subset \bR^m_+$ which contains a fundamental domain for $\bar{\Lam}$. Given 
a fundamental domain $F'$ for $\Lam$ such that $\bar{F}'\subset K'$, we have that $A=F'\cap\bZ^m$ is a set of representatives
of $\bZ^m/\Lam$.  Denoting $C=\set{\sum_1^mt_i\mb{e}_i:0\le t_i <1},$ 
the standard fundamental domain of $\bZ^m$, it is straightforward to show that the disjoint union of cubes 
$F=\cup_{\mb{k}\in A}(\mb{k}+C)$ is again a fundamental domain of $\Lam$. 
It follows  
that $\bar{F}=\cup_{\mb{k}\in A}(\bar{\mb{k}} +\bar{C})$ is a fundamental domain for 
$\bar{\Lam}$ and since $\bar{A}\subset K'$, by slightly enlarging $K'$ if necessary, we may
assume that $\bar{F}\subset K'$. 

On identifying $\bR^m/\bar{\Lam}$ with $\bar{F}$ we see that the support of $\bar{\nu}_\Lam$ is exactly the set $\bar{A}$ and for $\mb{k}\in A$, the weight 
$\bar{\nu}_\Lam(\bar{\mb{k}})=\av{A}^{-1} = \lam(\bar{C})$. 
Thus, by expressing the integrals on the left hand side of~\eqref{eq:1131} as a sum over the cubes $\set{\bar{\mb{k}}+\bar{C}:\mb{k}\in A}$, using 
the triangle inequality, and inserting the absolute value into the integral, we see that the expression on the left hand side of~\eqref{eq:1131} is bounded above by
\begin{align}\label{eq:1219}
&\sum_{\mb{k}\in A}\int_{\bar{C}} \av{\on{d}_{\bR^m/\bar{\Lam}}^\al(0,\bar{\mb{k}}+ x) -\on{d}_{\bR^m/\bar{\Lam}}^\al(0,\bar{\mb{k}})}d\lam(x)\\
\nonumber 
 &\le  \sup_{x\in \bar{F}} \sup_{ y\in x+\bar{C}} \av{\on{d}_{\bR^m/\bar{\Lam}}^\al(0,x)-\on{d}_{\bR^m/\bar{\Lam}}^\al(0,y)}.
\end{align}
By Lemma~\ref{equicontinuous}, as $\bar{\Lam}$ varies, the family of functions 
$\on{d}_{\bR^m/\bar{\Lam}}^\al(0,x)$ is equicontinuous on $K'$ and therefore as the diameter of $\bar{C}$ goes to  $0$
the right hand side of~\eqref{eq:1219} goes to zero as well. As this diameter goes to $0$ as $\av{\Lam}\to\infty$ the claim follows.
\end{proof}
We now collect the above preparations and isolate an important component in the proof of Theorem~\ref{main theorem 1}. 
\begin{corollary}\label{cor:1109}
Let $\cA_i$ be finite collections of subgroups of $\bZ^m$ of index $\ell_i$. Let $\mu_i$ denote the normalized counting
measure on $\bar{\cA}_i=\set{\bar{\Lam}\in X_m:\Lam\in\cA_i}$
and assume $\mu_i\wstar m_{X_m}$.
Choose $\xi$ as in~\eqref{cdiam}, \eqref{cavgdist}, or \eqref{cgirth} and let $\xi_i:\bar{\cA}_i\to(0,\infty)$ be defined as $\xi_i(\bar{\Lam})= \ell_i^{-\frac{1}{m}} \xi(\cG)$,
where $\cG=C_{\bZ^m/\Lam}(I)$ or $C_{\bZ^m/\Lam}^+(I)$ according to the case. Then, if $\zeta:X_m\to(0,\infty)$ is as in~\eqref{covrad}, \eqref{avgdisttorus}, or \eqref{girthtorus} according to the choice of $\xi$,  
then  $(\xi_i)_*\mu_i\wstar \zeta_*m_{X_m}$.
\end{corollary}
\begin{proof}
The proof in each case is an application of Lemma~\ref{lem:pushing convergence}. 
%
%
%
%
Our task is therefore reduced to verifying 
the conditions in this lemma in each case.
The first condition of Lemma~\ref{lem:pushing convergence} is properness of $\zeta$. 
It is straightforward to deduce from Mahler's compactness criterion that if $\zeta$ is as in~\eqref{covrad}-\eqref{girthtorus}, then 
$\zeta$ is continuous and proper\footnote{
Note that the properness refers to the range $(0,\infty)$ and in fact in case \eqref{girthtorus} $\zeta(\Lam)$ can approach $0$
as $\Lam$ develops short vectors.}
on $X_m$. 
The second condition of Lemma~\ref{lem:pushing convergence} holds by Lemma~\ref{lem:quazi isometric} in cases~\eqref{cdiam}, \eqref{cgirth} 
(with no reference to a compact set in $X_m$),
and by Lemma~\ref{lem:moments} in case~\eqref{cavgdist}.
\end{proof}

\section{Random Cayley graphs are approximate torus graphs}\label{sec:main prop}
In this section we prove Theorem~\ref{main theorem 1}.
Let $\Sig<\bZ^n$ be a finite index subgroup and consider the set 
\begin{equation}\label{eq:2127}
\cA_\Sig = \set{\Lam<\bZ^m : \bZ^m/\Lam \simeq \bZ^n/\Sig}.
\end{equation}
There is a natural map $\tau : \cS_{\bZ^n/\Sig}(m)\to \cA_\Sig$ which is defined as follows:
Given $s\in\cS_{\bZ^n/\Sig}(m)$, let $\mb{u}\in\Mat_{n\times m}(\bZ)$
be a matrix whose columns represent the elements of $s$ and define $\vphi_{\mb{u}}:\bZ^m\to \bZ^n/\Sig$ to be the 
homomorphism induced by $\mb{u}$; that is, for $k\in \bZ^m,$ $\vphi_{\mb{u}}(k) = \mb{u}k +\Sig$. We define
\begin{align*}
\tau(s) = \Lam_s&\defi \ker \vphi_{\mb{u}} = \set{k\in \bZ^m: \mb{u}k\in\Sig}.
\end{align*}
It is clear that  $\Lam_s$ does not depend on the choice of $\mb{u}$ but only on the set $s$. 
Since $s$ is a generating
set for $\bZ^n/\Sig$, $\vphi_{\mb{u}}$ is onto and therefore descends to an isomorphism 
$\overline{\vphi}_{\mb{u}}:\bZ^m/\Lam_s\to \bZ^n/\Sig$. Hence, $\Lam_s\in \cA_\Sig$ as stated above. Moreover, since 
$\overline{\vphi}_{\mb{u}}$ takes the generating set $I$ of $\bZ^m/\Lam_s$ to the generating set $s$ 
of $\bZ^n/\Sig$ we 
conclude the following lemma which explains the importance of the lattices $\Lam_s$ to our discussion
as they help us transport the discussion from random generating sets to random approximated torus graphs. 
\begin{lemma}\label{lem:1928}
Let $\Sig$ and $s$ be as above.
\begin{enumerate}
\item\label{prop:11080} $\Lam_s$ is a
subgroup of index $\av{\Sig}$ of $\bZ^m$ and moreover there exists an isomorphism $\bZ^m/\Lam_s\to\bZ^n/\Sig$ which 
takes the generating set $I$ to $s$.
\item\label{prop:11084} The graph $C_{\bZ^n/\Sig}(s)$ is isomorphic to the approximated torus graph $C_{\bZ^m/\Lam_s}(I)$.
\item\label{prop:11085} The digraph $C_{\bZ^n/\Sig}^+(s)$ is isomorphic to the approximated torus digraph $C_{\bZ^m/\Lam_s}^+(I)$.
\end{enumerate}
\end{lemma} 
The object that will help us analyze the map $\tau: \cS_{\bZ^n/\Sig}(m)\to \cA_\Sig$ is $\Hom;$ the collection of 
homomorphisms from $\bZ^m$ onto $\bZ^n/\Sig$. We have the following natural diagram:
$$\xymatrix{ &\Hom\ar[dr]^{\pi_2} \ar[dl]_{\pi_1} &\\ \gen\ar[rr]_\tau & & \cA_\Sig.}$$
Here, $\pi_i$ are defined as follows: Given $\vphi\in\Hom$, $\pi_1(\vphi) = \vphi(I)$ (where $I$ is the standard basis of $\bZ^m$), and $\pi_2(\vphi) = \ker \vphi$. The map $\tau$ defined above takes $s\in\gen$ and chooses $\vphi$ in the 
$\pi_1$-fiber and then applies $\pi_2$. That is, $\tau$ closes the above diagram in a commutative manner.  On 
$\gen$ there is a natural action of $\on{Aut}(\bZ^n/\Sig)$ and on $\cA_\Sig$ there is a natural action of $\GL_m(\bZ)$. 
The main reason
we introduce $\Hom$ is that on it we have a natural action of the product $\GL_m(\bZ)\times \on{Aut}(\bZ^n/\Sig)$ such that 
each of $\pi_1$ and $\pi_2$ intertwines the action of the corresponding group. Namely, given $\vphi\in\Hom$ and 
$(\ga,\del)\in \GL_m(\bZ)\times \on{Aut}(\bZ^n/\Sig)$ we define $(\ga,\del)\vphi = \del \vphi \ga^{-1}$.

\begin{definition} 
We denote by $\nu_\Sig,\mu_\Sig, \bar{\mu}_\Sig$ the normalized counting measures
on $\cS_{\bZ^n/\Sig}(m),$ $\cA_\Sig,$ $\bar{\cA}_\Sig$ respectively. 
\end{definition}
\begin{lemma}\label{lem:2046}
\begin{enumerate}
\item\label{p151} $\GL_m(\bZ)\times \on{Aut}(\bZ^n/\Sig)$ acts transitively on $\Hom$.
\item\label{p152} $\GL_m(\bZ)$ acts transitively on $\cA_\Sig$.
\item\label{p153} 
The map $\tau:\cS_{\bZ^n/\Sig}(m)\to \cA_\Sig$ has fibers of constant size and so $\tau_*\nu_\Sig = \mu_\Sig$.
\item\label{p154} As $\Sig\to\infty$ according to Definition~\ref{def:divergence}, $\av{\cA_\Sig}\to\infty$.
\item\label{p155} As $\Sig\to\infty$ according to Definition~\ref{def:divergence}, $\bar{\mu}_\Sig\wstar m_{X_m}$.
\end{enumerate}
\end{lemma}
\begin{proof}
\eqref{p151}. Let $\vphi_1,\vphi_2\in \Hom$ and assume first that $\ker\vphi_1=\ker\vphi_2 = \Lam$, 
if we denote by $\ol{\vphi}_i:\bZ^m/\Lam\to \bZ^n/\Sig$ the corresponding isomorphisms that $\vphi_i$ descend to, then
post-composing $\vphi_2$ with $\del = \ol{\vphi}_1\ol{\vphi}_2^{-1}\in\on{Aut}(\bZ^n/\Sig)$ we get that $\del \vphi_2$ and $\vphi_1$ have the same kernel $\Lam$ and descend to the same map on $\bZ^m/\Lam$. We conclude that 
$\del\vphi_2 = \vphi_1$. Thus in order to complete the proof of \eqref{p151} we only need to show that given $\vphi_1,\vphi_2\in \Hom$
there exists $\ga\in \GL_m(\bZ)$ such that $\ker \vphi_1 = \ker \vphi_2\ga = \ga^{-1}\ker\vphi_2$. 
Indeed, if $\ker\vphi_i = \Lam_i$ then $\bZ^m/\Lam_i$ are isomorphic and since $\bZ^m$ is a free abelian group such an isomorphism must arise from an isomorphism $\bZ^m\to\bZ^m$ taking $\Lam_1$ to $\Lam_2$ which completes the proof.

\noindent\eqref{p152}. Since $\pi_2$ is onto and  intertwines the actions of $\GL_m(\bZ)$, and since 
for $\del\in \on{Aut}(\bZ^n/\Sig)$ and  $\vphi\in \Hom$ we have that $\ker \vphi = \ker \del\vphi$, the transitivity proved 
in \eqref{p151} implies the transitivity claimed in \eqref{p152}. 

\noindent\eqref{p153} The fibers of $\pi_1$ are of size $m!$ and each such fiber is contained in a single $\pi_2$-fiber, thus 
the claim will follow once we establish that $\pi_2$ has fibers of constant size. Indeed, the transitivity proved in~\eqref{p151} and \eqref{p152}, and the fact (which was used in the proof of~\eqref{p152}) saying 
that the $\on{Aut}(\bZ^n/\Sig)$-action preserves the $\pi_2$-fibers imply that given two $\pi_2$-fibers  
there exists an element of $\GL_m(\bZ)$ which maps one to the other. In particular, they are of the same size as claimed.

\noindent\eqref{p154} Recall that given a sublattice $\Lam<\bZ^m$ there exist unique positive integers $\ell_1,\dots,\ell_m$ referred to as \textit{the elementary divisors} of $\Lam$ in $\bZ^m$, which satisfy the division relation $\ell_i |\ell_{i+1}$ and for which there exists
a basis $\set{v_i}_1^m$ of $\bZ^m$ such that $\set{\ell_i v_i}_1^m$ form a basis for $\Lam$ (see for example~\cite{SD01}).
It is shown in the proof of~\cite[Theorem 5.3]{EMSS} that the cardinality of the orbit $\GL_m(\bZ)\Lam$, which equals $\av{\cA_\Sig}$
by~\eqref{p152}, goes to infinity
if and only if the ratio $\ell_m/\ell_1$ does so too.
Due to the division relation between the $\ell_i$'s, the latter is equivalent to the divergence of
$(\prod_1^m\ell_i)/\ell_1^m = \av{\Lam}/\ell_1^m = \av{\Sig}/\ell_1^m.$ We are therefore reduced to showing that as $\Sig\to\infty$
according to Definition~\ref{def:divergence}, one has $\av{\Sig}/\ell_1^m\to\infty$. 
Since $\bZ^n/\Sig\simeq \bZ^m/\Lam$ and $m\ge n$ we have that the elementary divisors of $\Lam$ in $\bZ^m$ are obtained from 
to those of $\Sig$ in $\bZ^n$ by augmenting them by $m-n$ 1's in the beginning. Thus, if $m>n$ then $\ell_1=1$ and so $\Sig\to\infty$ if and only if $\av{\Sig}/\ell_1^m\to\infty.$ In the case $n=m$, since $\ell_1$ is the first elementary divisor of $\Sig$ in $\bZ^n$, then we get that $\ell_1=\gcd \Sig$. Hence, again $\Sig\to\infty$ if and only if $\av{\Sig}/\ell_1^n\to\infty$, which concludes the proof
of \eqref{p154}.

\noindent \eqref{p155}. This follows from part~\eqref{p154} and the following theorem which is a special case of 
\cite[Theorem 1.2]{EO06} (see also \cite{COU01}, \cite{GM03}):
\begin{theorem}\label{thm:eqdist}
If $\mu_i$ is the normalized counting measure supported on a finite $\GL_m(\bZ)$-orbit in $X_m$ such that the cardinality $\av{\supp(\mu_i)}\to\infty$ as $i\to\infty$, then $\mu_i\wstar m_{X_m}$.
\end{theorem}

\end{proof}
We are now ready to prove the main result, Theorem~\ref{main theorem 1}.
\begin{proof}[Proof of Theorem~\ref{main theorem 1}]
Choose $\Sig_i\to\infty$ and let $\xi$ be as in~\eqref{cdiam}, \eqref{cavgdist}, or \eqref{cgirth}.
Consider the diagram
$$\xymatrix{
\cS_{\bZ^n/\Sig_i}(m) \ar[dr]_{\xi_i}\ar[rr]^{s\mapsto\bar{\tau}(s) =\bar{\Lam}_s} &&\bar{\cA}_{\Sig_i}\ar[dl]^{\xi_i'}\\
& \bR&
}$$
where $\xi_i(s)  = \av{\Sig_i}^{-\frac{1}{m}}\xi(\cG)$, where $\cG=C_{\bZ^n/\Sig_i}(s)$ or $C_{\bZ^n/\Sig_i}^+(s)$ according to the case, and similarly,  $\xi_i'(\bar{\Lam})  = \av{\Sig_i}^{-\frac{1}{m}}\xi(\cG)$, where $\cG = C_{\bZ^m/\Lam}(I)$ or $C_{\bZ^m/\Lam}^+(I)$ according to the case. By Lemma~\ref{lem:1928} the diagram is indeed commutative.
Our goal is to show that $\xi_i$ converge in distribution to the random variable $\zeta$ on $X_m$, where $\zeta$ is as in
\eqref{covrad}, \eqref{avgdisttorus}, or \eqref{girthtorus} according to the choice of $\xi$. 
This is the same as saying that $(\xi_i)_*\nu_{\Sig_i}\wstar \zeta_* m_{X_m}$.  Since  $\xi_i = \xi_i'\comp \bar{\tau}$ and
by Lemma~\ref{lem:2046}\eqref{p153} $\tau_*\nu_{\Sig_i} = \mu_{\Sig_i}$ (which is equivalent to saying that 
$\bar{\tau}_*\nu_{\Sig_i} = \bar{\mu}_{\Sig_i}$), we are reduced to showing that $(\xi_i')_*\bar{\mu}_{\Sig_i}\wstar \zeta_* m_{X_m}$.
This follows from Corollary~\ref{cor:1109} which is applicable since $\bar{\mu}_{\Sig_i}\wstar m_{X_m}$ by Lemma~\ref{lem:2046}\eqref{p155}. 
\end{proof}

\section{A refinements of Theorem~\ref{main theorem 1}}\label{refinements}
In this section we restrict attention to the case $m>n$ and consider only groups of the form $\bZ^n/\Sig$ where $\Sig=k\bZ^n$. Thus according to 
Definition~\ref{def:divergence}, $\Sig\to\infty$ if and only if $k\to\infty$.
We wish to define families of natural non-uniform probability measures on $\cS_{\bZ^n/\Sig}(m)$ with respect to which the conclusion of Theorem~\ref{main theorem
1} will remain valid.

Our next goal is to put all the sets $\cS_{\bZ^n/k\bZ^n}(m)$ in a fixed ambient space.
As a first step, our restriction to $\Sig=k\bZ^n$ allows us to embed all the groups $\bZ^n/\Sig$ in a single continuous torus. 
Namely, let $\bT^n=\bR^n/\bZ^n$ and let $\iota_k: \bZ^n/k\bZ^n\hookrightarrow \bT^n$ be defined by $\iota_k(u+k\bZ^n) = \frac{1}{k}u+\bZ^n$. 
Next, the collection of subsets of size $m$ in $\bT^n$ may be thought of as contained in the quotient of the product $(\bT^n)^m$ by the permutation 
group $S_m$ on $\set{1,\dots,m}$. Let us denote this quotient by $Y= S_m\backslash (\bT^n)^m$. 
Thus, using $\iota_k$ we may and will identify $\cS_{\bZ^n/k\bZ^n}(m)$ with a subset of $Y$. 
We shall refer to the push-forward probability measure $\pi_*\lam^{\otimes m}$ on $Y$ as the  Lebesgue measure on $Y$, 
where $\pi:(\bT^n)^m\to Y$ is the quotient map and $\lam$ is the usual Lebesgue measure on the torus $\bT^n$. 

Given a subset  
$D\subset Y$, we define the probability
measure $\nu_{D,k}$ on $\cS_{\bZ^n/k\bZ^n}(m)$ as the restriction of the uniform probability measure on $\cS_{\bZ^n/k\bZ^n}(m)$ to $D$ 
(assuming that $\cS_{\bZ^n/k\bZ^n}(m)\cap D\ne\varnothing$). Finally, by a Jordan measurable set $D\subset Y$ we mean
a measurable set 
with boundary of Lebesgue measure zero. The importance of Jordan measurability to our discussion is that 
the characteristic function $\chi_D$ of such a set can be approximated from above and below (in $L^1$-norm) by continuous functions and thus $\chi_D$ behaves nicely when integrated against a sequence of measures which converge weak-star to the Lebesgue measure. 
\begin{theorem}\label{main theorem 2}
Fix $m>n$ and let $\Sig<\bZ^n$ be of the form $\Sig = k\bZ^n$.
Let $D\subset Y$ be a Jordan measurable set of positive Lebesgue measure and let $\nu_{D,k}$ be the restriction of the uniform 
measure on $\cS_{\bZ^n/\Sig}(m)$ to $D$. 
Then, for all $k$ large enough $\cS_{\bZ^n/k\bZ^n}(m)\cap D\ne\varnothing$ and 
the conclusion of Theorem~\ref{main theorem 1} remains valid (with the same limit distributions) if $s$ is chosen randomly according to $\nu_{D,k}$.
\end{theorem} 
\begin{proof}[Sketch of proof]
The proof is identical to that of Theorem~\ref{main theorem 1} where instead of using the collection 
$\cA_\Sig=\set{\Lam_s:s\in\cS_{\bZ^n/\Sig}(m)}$ one uses the sub-collection $ \set{\Lam_s: s\in\cS_{\bZ^n/\Sig}(m)\cap D}.$ 
The only ingredient 
missing is an equidistribution result substituting Lemma~\ref{lem:2046}\eqref{p155} saying that the normalized counting measure 
$\bar{\mu}_{D,k}$ on $\set{\bar{\Lam}_s: s\in\cS_{\bZ^n/k\bZ^n}(m)\cap D}$ satisfies $\bar{\mu}_{D,k}\wstar m_{X_m}$. 

In order to justify this equidistribution statement we need to introduce some terminology and notation. 
For $s\in \cS_{\bZ^n/\Sig}(m)$ let $\mb{u}\in\Mat_{n\times m}(\bZ)$ be a matrix whose columns represent the elements of $s$. 
Let $d=n+m$ and consider the matrix $g_{\mb{u}} = \smallmat{I_m & 0 \\ \mb{u} &kI_n}$. It is clear that the lattice $L_s = g_{\mb{u}}\bZ^d$
does not depend on the choice of $\mb{u}$ but only on $s$. We leave it as an exercise (see for 
example \cite[equation (2.12)]{EMSS}) to check that 
$L_s\cap \bR^m = \Lam_s$ (where $\bR^m$ denotes the subspace of the first $m$ coordinates in $\bR^d$).
Let us denote by
$L_s'\in X_d$  the unimodular lattice obtained by applying to the lattice $L_s$ in $\bR^d$ the diagonal matrix 
$a(k)=\smallmat{k^{-\frac{n}{m}}I_m&0\\0&I_n}$, so that by the above exercise $L_s'\cap\bR^m = \bar{\Lam}_s$.
The equidistribution $\bar{\mu}_{D,k}\wstar m_{X_m}$ 
follows from ~\cite[Theorem 1.3]{EMSS} which
is a joint equidistribution theorem which may be restated (see Remark~\ref{final remark}) as saying that the collection of pairs 
\begin{equation}\label{eq:emss}
\set{(s,L_s'):s\in\cS_{\bZ^n/k\bZ^n}(m)}
\end{equation}
equidistributes
in the product space $Y\times Z$.
Here, $Y$ is as above and  
$Z\subset X_d$ is the subspace defined by $Z= \set{L\in X_d:L\cap \bR^m\in X_m}$. The space $Z$ has a natural 
`uniform' probability measure on it and the equidistribution eluded to above means that the normalized counting measure
on the collection~\eqref{eq:emss} converges weak-star to the product of the Lebesgue measure on $Y$ and the uniform measure 
on $Z$. In turn, there is a natural projection $Z\to X_m$ defined by $L\mapsto L\cap \bR^m$ and since the uniform measure on
$Z$ projects to $m_{X_m}$ we conclude that the collection of pairs $\set{(s,\bar{\Lam}_s):s\in\cS_{\bZ^n/k\bZ^n}(m)}$
equidistributes in the product $Y\times X_m$. In particular,
the fact that the limit measure is a product measure implies that if we condition on the left coordinate being in $D$ the right coordinate still equidistributes in $X_m$. 
That is, $\set{\bar{\Lam}_s:s\in \cS_{\bZ^n/k\bZ^n}(m)\cap D}$ equidistributes in $X_m$ as desired.
The assumption that $D$ is Jordan measurable and of positive measure is exactly the condition needed in order that the equidistribution in the product $Y\times Z$ may
be conditioned on the fact that the left coordinate is in $D$. 
\end{proof}
\begin{remark}\label{final remark}
Strictly speaking~\cite[Theorem 1.3]{EMSS} does not deal with the product space $Y\times Z$ but with $(\bT^n)^m\times Z$.  
Instead of considering subsets $s\in\cS_{\bZ^n/k\bZ^n}(m)$ they consider $m$-tuples $\mb{u}=\set{u_1,\dots ,u_m}\in (\bZ^n/k\bZ^n)^m$ (which we think of as contained in $(\bT^n)^m$ using $\iota_k$), where the $u_i$'s are assumed to generate $\bZ^n/k\bZ^n$. 
This  difference entails a minor issue which is that $m$-tuples corresponds to subsets of size $\le m$ rather than equal to $m$ and furthermore, this correspondence is not 1-1 and the size of the fibers is not constant. Since the percentage of $m$-tuples which correspond to sets of size $<m$ is negligible these two issues may be ignored and it is 
straightforward to deduce the version of the theorem stated above from the formulation in~\cite{EMSS}. We leave the details to the interested reader.
\end{remark}

\begin{bibdiv}
\begin{biblist}

\bib{AGG}{article}{
      author={Amir, Gideon},
      author={Gurel-Gurevich, Ori},
       title={The diameter of a random {C}ayley graph of {$\Bbb Z_q$}},
        date={2010},
        ISSN={1867-1144},
     journal={Groups Complex. Cryptol.},
      volume={2},
      number={1},
       pages={59\ndash 65},
         url={http://dx.doi.org/10.1515/GCC.2010.004},
      review={\MR{2672553}},
}

\bib{COU01}{article}{
      author={Clozel, Laurent},
      author={Oh, Hee},
      author={Ullmo, Emmanuel},
       title={Hecke operators and equidistribution of {H}ecke points},
        date={2001},
        ISSN={0020-9910},
     journal={Invent. Math.},
      volume={144},
      number={2},
       pages={327\ndash 351},
         url={http://dx.doi.org/10.1007/s002220100126},
      review={\MR{1827734}},
}

\bib{EMSS}{article}{
      author={Einsiedler, Manfred},
      author={Mozes, Shahar},
      author={Shah, Nimish},
      author={Shapira, Uri},
       title={Equidistribution of primitive rational points on expanding
  horospheres},
        date={2016},
        ISSN={0010-437X},
     journal={Compos. Math.},
      volume={152},
      number={4},
       pages={667\ndash 692},
         url={http://dx.doi.org/10.1112/S0010437X15007605},
      review={\MR{3484111}},
}

\bib{EO06}{article}{
      author={Eskin, Alex},
      author={Oh, Hee},
       title={Ergodic theoretic proof of equidistribution of {H}ecke points},
        date={2006},
        ISSN={0143-3857},
     journal={Ergodic Theory Dynam. Systems},
      volume={26},
      number={1},
       pages={163\ndash 167},
         url={http://dx.doi.org/10.1017/S0143385705000428},
      review={\MR{2201942}},
}

\bib{GM03}{article}{
      author={Goldstein, Daniel},
      author={Mayer, Andrew},
       title={On the equidistribution of {H}ecke points},
        date={2003},
        ISSN={0933-7741},
     journal={Forum Math.},
      volume={15},
      number={2},
       pages={165\ndash 189},
         url={http://dx.doi.org/10.1515/form.2003.009},
      review={\MR{1956962}},
}

\bib{M10}{article}{
      author={Marklof, Jens},
       title={The asymptotic distribution of {F}robenius numbers},
        date={2010},
        ISSN={0020-9910},
     journal={Invent. Math.},
      volume={181},
      number={1},
       pages={179\ndash 207},
         url={http://dx.doi.org/10.1007/s00222-010-0245-z},
      review={\MR{2651383}},
}

\bib{MS11}{article}{
      author={Marklof, Jens},
      author={Str{\"o}mbergsson, Andreas},
       title={Diameters of random circulant graphs},
        date={2013},
        ISSN={0209-9683},
     journal={Combinatorica},
      volume={33},
      number={4},
       pages={429\ndash 466},
         url={http://dx.doi.org/10.1007/s00493-013-2820-6},
      review={\MR{3133777}},
}

\bib{SD01}{book}{
      author={Swinnerton-Dyer, H. P.~F.},
       title={A brief guide to algebraic number theory},
      series={London Mathematical Society Student Texts},
   publisher={Cambridge University Press, Cambridge},
        date={2001},
      volume={50},
        ISBN={0-521-00423-3},
         url={http://dx.doi.org/10.1017/CBO9781139173360},
      review={\MR{1826558}},
}

\end{biblist}
\end{bibdiv}
\def\cprime{$'$} \def\cprime{$'$} \def\cprime{$'$}

\end{document}